\documentclass[11pt]{amsart}
\usepackage{graphicx}
\usepackage{amssymb}

\evensidemargin
\oddsidemargin
\topskip
\abovedisplayskip
\belowdisplayskip
\footskip

\def\cod{\mbox{\rm codim }}

\def\C{\mathbb C}
\def\R{\mathbb R}
\def\Ker{\mbox{Ker\hspace{0.04cm}}}
\def\Im{\mbox{Im\hspace{0.04cm}}}

\newtheorem{teo}{\sc Theorem}[section]
\newtheorem{prop}[teo]{\sc Proposition}
\newtheorem{lem}[teo]{\sc Lemma}
\newtheorem{cor}[teo]{\sc Corollary}
\newtheorem{dfn}[teo]{\normalfont \em Definition}
\newtheorem{ex}[teo]{\normalfont \em Example}
\newtheorem{remark}[teo]{\normalfont \em Remark}

\begin{document}

\title{Special Jordan Subspaces in Coupled Cell Networks}

\author[C.S. Moreira]{C\'elia Sofia Moreira}

\address{CMUP - Centre of Mathematics, University of Porto -  Rua do Campo Alegre 687, 4169-007 Porto}

\email{celiasofiamoreira@gmail.com}

\maketitle

\begin{center}{\today}\end{center}

\begin{abstract} Given a regular network (in which all cells have the same type and receive the same number of inputs and all arrows have the same type), we define the {\em special Jordan subspaces} to that network and we use these subspaces to study the synchrony phenomenon in the theory of coupled cell networks. To be more precise, we prove that the synchrony subspaces of a regular network are precisely the polydiagonals that are direct sums of special Jordan subspaces. We also show that special Jordan subspaces play a special role in the lattice structure of all synchrony subspace because every join-irreducible element of the lattice is the smallest synchrony subspace containing some special Jordan subspace.
\end{abstract}

\vspace{0.2cm}

\section{Introduction}

\subsection{Regular coupled cell networks} In this subsection we briefly recall a few facts concerning the theory of (regular) coupled cell networks  developed by Stewart, Golubitsky and coworkers \cite{SGP03,GST05}.

A {\em cell} is a system of ordinary differential equations and a {\em coupled cell system} is a finite collection of interacting cells. A coupled cell system can be associated with a {\em network}, a directed graph whose nodes represent cells and whose arrows represent couplings between cells. The general theory allows for {\em loops} and {\em multiple arrows}. All couplings of the same type between two cells are represented by a single arrow with the number of couplings attached to it, unless this number is equal to 1, in which case it is simply omitted. The general theory associates to each network a class of admissible vector fields, consistent with the structure of the network.

In this paper we restrict attention to {\em regular networks}, that is, networks associated with coupled cell systems where all cells have the same differential equation (up to reordering of coordinates) and one kind of coupling. In this case, the state spaces of the cells are all identical, say an euclidean space $\R^k$, with $k\ge 1$, and so, if the network has $n$ cells then the {\em total phase space} is $(\R^k)^n$. The {\em valency} of a regular network is the number of arrows that input to each cell. 

A {\em polydiagonal} is a subspace of the total phase space that is defined by the equalities of certain cell coordinates. The total phase space is polydiagonal and the {\em fully synchrony subspace} is the polydiagonal defined by the equalities of all coordinates. A {\em synchrony subspace} is a polydiagonal that is flow invariant for every admissible vector field. 
Every regular network has at least two synchrony subspaces: the fully synchrony subspace and the total phase space.  In this work, these are called the {\em trivial} synchrony subspaces.

Golubitsky {\em et al.}~ \cite{GST05} proved that every coupled cell system associated with a network when restricted to a synchrony subspace corresponds to a coupled cell system associated with a smaller network, called the {\em quotient network}. 

The {\em adjacency matrix} of a regular network is a square matrix $A=[a_{ij}]$, where $a_{ij}$ is the number of arrows that cell $i$ receives from cell $j$. Note that each row sum of the adjacency matrix equals the valency of the network.

Since every adjacency matrix is a multiple of a stochastic matrix, well-known results about stochastic matrices can be applied to adjacency matrices and guarantee that:
\begin{enumerate}
\item the valency $v$ is a semi-simple eigenvalue (its algebraic and geometric multiplicities coincide) of the adjacency matrix and, for every eigenvalue $\lambda$ of this matrix, $|\lambda|\le v$.
\item $(1,\cdots,1)$ is an eigenvector associated with the valency.
\end{enumerate}

\subsection{Synchrony}

In the theory of coupled cell networks the concept of {\em synchrony} has always played a special role.
In fact, synchrony is an important phenomenon for networks in general, and in the last years it has been subject of several research studies in many different areas, such as Internet, spread of epidemics, food webs in ecosystems, neural circuits, gene transcription and cellular signaling \cite{GS06}.

Stewart, Golubitsky and coworkers \cite{GST05,SGP03} proved that it is possible to identify synchrony patterns in a network using solely the network architecture. It is really amazing and surprising that we can identify synchrony in a network just by analyzing its structure, without knowing anything about the internal dynamics of each cell and the specific equations of the admissible coupled cell systems. Moreover, Stewart proved that the set of all synchrony subspaces of a regular network is a complete lattice \cite{S09}. 

In this sense, it is important to study possible structures and properties of the set of all synchrony patterns in regular networks. We are lead to important questions concerning synchrony in networks, such as: why does a network has a small or great number of synchrony subspaces? Where does the synchrony come from? What are the important features of the corresponding lattice of synchrony subspaces?

\subsection{Synchrony and Jordan subspaces}

Theorem 4.3 of Golubitsky {\em et al.} \cite{GST05} is a powerful tool in the study of synchrony subspaces because it allows to relate synchrony subspaces of regular networks to invariant subspaces of the corresponding adjacency matrix $A$, in the following way: a subspace of the total phase space is a synchrony subspace if and only if it is polydiagonal and $A$-invariant, assuming that each cell phase space is $\R$.

Moreover, in Corollary 2.10 and Remark 5.3 of Aguiar and Dias~\cite{AD12} it is carefully explained that in the calculation of synchrony subspaces each cell phase space can be extended from $\R$ to $\C$. Basically and briefly (for more details see \cite{AD12}), this is due to the fact that a vector $v=(v_1,\cdots,v_n)\in \C^n$ satisfies an equality of coordinates $x_i=x_j$ if and only if $$\hspace*{1.5cm}\mbox{Re} (v_i)=\mbox{Re}(v_j)\, \mbox{ and }\,  \mbox{Im}(v_i)=\mbox{Im}(v_j), \quad 1\le i<j\le n.$$

On another hand, given a linear transformation $A$ from $\C^n$ into $\C^n$, it is well-known that a subspace is $A$-invariant if and only if it is a direct sum of Jordan subspaces \cite{LT}. Thus, we obtain the following result that will play a special role along all this work and that can be interpreted as a corollary of Theorem 4.3 of Golubitsky {\em et al.} \cite{GST05}:

\begin{lem}\: \label{known}Given an $n$-cell regular network, a subspace of the total phase space is a synchrony subspace if and only if it is polydiagonal and a direct sum of Jordan subspaces of the corresponding adjacency matrix, assuming that each cell phase space is $\C$.
\end{lem}

In particular, if the adjacency matrix of a regular network has only simple eigenvalues then all synchrony subspaces are polydiagonal and direct sums of eigenspaces. For example, the adjacency matrix of the regular network in Figure \ref{FigSE} has only simple eigenvalues and the corresponding eigenspaces are:
$$\begin{array}{l} G_2=\{x_1=x_2=x_3=x_4\},\\G_0= \{x_1=x_4, x_2=x_3, x_1+x_2=0\},\\
G_{1}=\{x_1=x_3,  x_1+x_2=0, x_4=0\},\\G_{-1}= \{x_1=x_2=x_3, 2x_1+x_4=0\}.
\end{array}$$

In this case, there are only four nontrivial synchrony subspaces:
$$\begin{array}{ll}
1. \; G_2 \oplus G_0   \,=\, \{x_1=x_4,\, x_2=x_3\} , & 3.\;   G_2 \oplus G_1\oplus G_{-1} \,=\,  \{x_1=x_3\},\\
2. \; G_2 \oplus G_{-1}  \,=\, \{x_1= x_2=x_3\}, & 4. \; G_2 \oplus  G_0 \oplus G_{-1} = \{x_2=x_3\}. \end{array}$$

 \begin{figure}[!hb]
 
\begin{minipage}[b]{0.4\linewidth}\centering
        \includegraphics[scale=0.37]{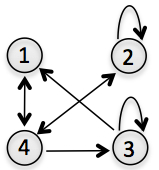} \vspace{-0.8cm}
 \end{minipage}  \begin{minipage}[b]{0.4\linewidth} \centering
$$\left(\begin{array}{cccc}
0 & 0 & 1 & 1  \\
0 & 1 & 0 & 1  \\
0 & 0 & 1 & 1  \\
1 & 1 & 0 & 0  \end{array}\right)$$
 \end{minipage}
 
 \caption{Regular network (left) whose adjacency matrix (right) has only simple eigenvalues.} \label{FigSE}
 \end{figure}

This case of simple eigenvalues is special because, for a regular network, all synchrony subspaces can be obtained simply by computing all possible sums (finite and direct) of such eigenspaces. We can use this fact to make some important observations about the lattices of these networks. For example, when $n\ge 4$, we can guarantee the existence of some possible lattice structures such as pentagons. H. Kamei \cite{K09} proved that when $n=4$ there is only one possible pentagon structure, which has two vertices as 2-dimensional synchrony subspaces and one vertex as a 3-dimensional synchrony subspace. It is clear that when $n$ increases, the number of possible pentagon structures also increases and, for example, for $n=5$ there are four possible pentagon structures which are shown in Figure \ref{Pentagon}.

\begin{figure}[!htb] \includegraphics[scale=0.47]{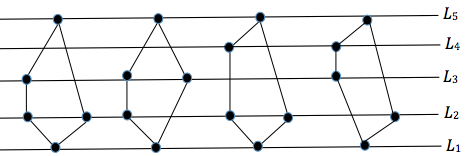}\caption{All possible pentagon lattices for 5-cell regular networks whose adjacency matrices have only simple eigenvalues. Level $L_i$ contains $i$-dimensional synchrony subspaces, with $1\le i\le 5$.}\label{Pentagon}\end{figure}

\subsection{Motivation}

\subsubsection{Motivating References}

The above relationship between synchrony and Jordan subspaces obtained from Theorem 4.3 of Golubitsky {\em et al.} \cite{GST05} was detailed studied by H.~Kamei for the special case of simple eigenvalues in \cite{K09} and this work was the main motivation of our paper. More recently, Aguiar and Dias \cite{AD12} extended this work to general networks and provided an algorithm to construct the lattice of all synchrony subspaces through a small set of synchrony subspaces. They also proved that the problem of obtaining the lattice of synchrony subspaces of a general network can be reduced to the problem of obtaining the lattice of synchrony subspaces of regular networks.

\subsubsection{Special Jordan subspaces}
Lemma \ref{known} guarantees that all synchrony subspaces of a regular network can be written as a direct sum of Jordan subspaces of the corresponding adjacency matrix. Since synchrony subspaces are defined by equalities of coordinates, the knowledge of all synchrony subspaces can be obtained from the list of all Jordan subspaces satisfying at least one equality of coordinates. Nevertheless, some observations in concrete examples made us believe that some Jordan subspaces were essential in this list and some others were not. Then, we conjectured the existence of a class of Jordan subspaces sufficient to generate all synchrony subspaces by direct sums. For example, the regular network in Figure \ref{Fig+-i} has only five nontrivial synchrony subspaces, namely, $$\begin{array}{l} S_1= \{x_1=x_2=x_3, x_4=x_5\}, \\S_2= \{x_1=x_4=x_5,  x_2=x_3\},\\
S_3= \{x_2=x_3=x_4=x_5\},\\S_4= \{x_2=x_3, \, x_4=x_5\}, \\ S_5=\{x_1=x_4\}.\end{array}$$

Its adjacency matrix has four distinct eigenvalues, namely, 2, $-1$ and $\pm i$, and the corresponding eigenspaces are:
$$\begin{array}{l}G_2= \{x_1=x_2=x_3=x_4=x_5\},\\ G_{-1}= \{x_2=x_3, \, x_4=x_5, x_1+x_2+x_4=0\},\\
G_{\pm i}=\, \{x_1=x_4,  2x_1+2x_2+x_3=0, \, \pm (2 i) x_1+x_3=0, 3x_1+x_2+x_5=0\}.\end{array}$$

\begin{figure}[!hb]
\begin{minipage}[b]{0.37\linewidth}\centering
        \includegraphics[scale=0.37]{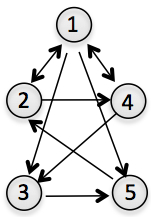} \vspace{-1.1cm}
 \end{minipage}  \begin{minipage}[b]{0.4\linewidth} \centering
$$\left(\begin{array}{ccccc}
0 & 1 & 0 & 1 & 0 \\
1 & 0 & 0 & 0 & 1  \\
1 & 0 & 0 & 1 & 0 \\
1 & 1 & 0 & 0 & 0\\
1 & 0 & 1 & 0 & 0 \end{array}\right)$$ 
 \end{minipage}
 
 \caption{Regular network and the corresponding adjacency matrix.} \label{Fig+-i}
 \end{figure}

Observing that 
$$\begin{array}{l}
S_i=G_2\oplus W_i,\; \mbox{for } 1\le i \le 3,\\
S_4= G_2 \oplus W_1\oplus  W_2,\\
S_5= G_2 \oplus W_2 \oplus G_i \oplus G_{-i},\end{array}$$
where $W_1$, $W_2$ and $W_3$ are the following Jordan subspaces in $G_{-1}$, $$\begin{array}{l}
W_1=\{x_1=x_2=x_3, \, x_4=x_5, \, 2x_1+x_4=0\},\\
W_2= \{x_1=x_4=x_5, \, x_2=x_3, \, 2x_1+x_2=0\},\\
W_3 = \{x_2=x_3=x_4=x_5, \, x_1+2x_2=0\},\\
\end{array}$$
we easily understand that, among all Jordan subspaces in $G_{-1}$, the subspaces $W_1$, $W_2$ and $W_3$ are essential to list all nontrivial synchrony subspaces.

We prove that all synchrony subspaces of a general regular network can be obtained using the direct sum operation over a small set of Jordan subspaces --- the {\em special Jordan subspaces}. Recall that in Aguiar and Dias \cite{AD12} it is proved that  all synchrony subspaces of a general network can be obtained using the sum operation over a small set of synchrony subspaces --- the {\em sum-irreducible synchrony subspaces}. We emphasize that although our initial aim was to generalize the work of Kamei \cite{K09} to regular networks, our principal aim has become to evidence the close relationship between special Jordan subspaces and synchrony subspaces.

\subsection{Structure of the paper} In Section 2 we define {\em special subspaces} and we present all results about these subspaces that are extremely useful along the work. In Section 3 we introduce the key concept of this work, {\em special Jordan subspace} and we also prove Theorem \ref{FT}, which is the main result about special Jordan subspaces: given an $n$-cell regular network, there exists a direct sum decomposition of $\C^n$ into special Jordan subspaces. In Section 4 we prove Theorem \ref{FT2}, which is the main result of this work about synchrony and shows the special role played by special Jordan subspaces in the synchrony phenomenon: given a regular network, a subspace of the total phase space is a synchrony subspace if and only if it is polydiagonal and a direct sum of special Jordan subspaces. We also describe a method to list all synchrony subspaces. In Section 5 we study lattices consisting of all synchrony subspaces and we show that special Jordan subspaces play a special role in the structure of these lattices, being directly connected with the join-irreducible elements of such lattices, precisely, every join-irreducible element is the smallest synchrony subspace containing some special Jordan subspace.

\subsection*{Acknowledgments}

The author thanks Ana Paula Dias for all helpful comments and valuable suggestions. Research supported by the FCT (Funda\c c\~ao para a Ci\^encia e a Tecnologia) Grant SFRH/BPD/64844/2009 and partially funded by the European Regional Development Fund through the programme COMPETE and by the Portuguese Government through the FCT under the projects PEst-C/MAT/UI0144/2011 and PTDC/MAT/100055/2008.

\section{Special subspaces}

In this section we define special subspaces and we present important properties of these subspaces that will be used along this work.

\subsection{Definition}

Given a subspace $W$ of $\C^n$, we denote by $P(W)$ the smallest polydiagonal containing $W$. Notice the following basic properties of this operator, for all subspaces $W$ and $V$ of $\C^n$: 
\begin{enumerate}\item $W\subset P(W)$,\item $P(W\cap V)\subset P(W)\cap P(V)$, \item $W\subset V \, \Rightarrow \, P(W)\subset P(V)$,  \item $ W \mbox{ is polydiagonal } \Leftrightarrow P(W)=W$.\end{enumerate}

\begin{dfn} \normalfont Consider a subspace $E$ of $\C^n$. A subspace $W$ of $E$ is {\em special in} $E$ when, for every subspace $U$ of $E$,
$$[\dim U=\dim W , \;  P(U)\subset P(W)]\; \Rightarrow \; P(U)=P(W).$$
\end{dfn}

Thus, $W$ is special in $E$ when every subspace $U$ of $E$ having the same dimension and satisfying the same equalities of coordinates as $W$, does not satisfy any additional equality of coordinates.

\begin{ex} \label{ExSS}\normalfont Consider the following subspaces of $\C^5$:
$$\begin{array}{l} E=\{3x_1+4x_2-x_3+3x_5=0,\, x_2-x_3-x_4+x_5=0\},\\
 U=\{x_2=x_3,\, x_4=x_5=0, x_1+x_2=0\}, \\
 W=\{x_2=x_3=x_4=x_5, x_1+2x_2=0\}.\end{array}$$
Notice that $U$ and $W$ are both 1-dimensional subspaces of $E$. It is easy to see that $U$ is not special in $E$ because $P(W)\subsetneq P(U)$. It is also easy to understand that $W$ is special in $E$ because it satisfies three independent equalities of coordinates and there is only one 1-dimensional subspace satisfying four independent equalities of coordinates, namely, the fully synchrony subspace, which clearly is not contained in $E$.\end{ex}

\subsection{Results} In this subsection we present three important results about special subspaces.

\begin{prop} \label{W=}Consider a $k$-dimensional subspace $E$ of $\C^n$, with $k>1$. A nonzero subspace $W$ of $E$ is special in $E$ if and only if $W=E\cap P(W)$.\end{prop}

\begin{proof} The result is trivial when $W=E$ and so, henceforward we suppose that $W$ is a nonzero proper subspace of $E$. We assume firstly that $W=E\cap P(W)$ and we prove that $W$ is special in $E$. For every subspace $U$ of $E$ satisfying $$\dim U=\dim W,  \; P(U)\subset P(W),$$ we have $$ U\subset E\cap P(U)\subset E\cap P(W)=W.$$ So, $U=W$ and $P(U)=P(W)$, concluding that $W$ is special in $E$.

Reciprocally, we assume that $W$ is special in $E$ and we prove that $W=E\cap P(W)$. 
By contradiction, suppose that $W\ne E\cap P(W)$. Then, there is a $(k+1)$-dimensional subspace $\tilde W$ of $E\cap P(W)$. Because $\dim \tilde W>1$, there is at least a codimension-1 polydiagonal $X$ that does not contain $\tilde W$. So, $$\tilde W\cap X\subsetneq \tilde W\quad \mbox{ and }\quad  
\cod (\tilde W\cap X)= \cod \tilde W +1.$$
Notice that:\begin{enumerate}
\item $\dim  (\tilde W \cap X)  = n-(\cod \tilde W +1) = \dim \tilde W-1 = \dim W.$
\item $P(\tilde W\cap X)\subset P(\tilde W)\cap X\subsetneq P(\tilde W)\subset P(W)$ (recall that $\tilde W\subset P(W)).$\end{enumerate}
However, these two conditions contradict the fact that $W$ is special in $E$. Therefore, such a subspace $\tilde W$ does not exist and so, $W=E\cap P(W)$.\end{proof}

\begin{ex}\normalfont In Example \ref{ExSS}, $U$ is not special in $E$ because $U\neq E\cap P(U)$ and $W$ is special in $E$ because trivially $W=E\cap P(W)$.\end{ex}

Next we provide a characterization of special subspaces that will be very useful in the calculation of these subspaces. 

\begin{prop} \label{specialvector} Consider a $k$-dimensional subspace $E$ of $\C^n$, with $k>1$. A nonzero subspace $W$ of $E$ is special in $E$  if and only $W=E\cap X$, for some codimension-$\nu$ polydiagonal $X$, with $\nu=\dim E-\dim W$.\end{prop}

\begin{proof} The result is trivial when $W=E$ and so, henceforward we suppose that $W$ is a nonzero proper subspace of $E$. We assume firstly that $W=E\cap X$, for some polydiagonal $X$, and we prove that $W$ is special in $E$. Under this assumption, we have $P(W)\subset X$ and so,
$$W\subset E\cap P(W)\subset E\cap X=W.$$
Thus, $W=E\cap P(W)$ and, by Proposition \ref{W=}, $W$ is special in $E$.
    
Conversely, we assume that $W$ is special in $E$. Due to Proposition \ref{W=}, $W=E\cap P(W)$. The equalities of $P(W)$ in the subspace $E$ define an homogeneous linear system with $s$ equations, $s=\cod P(W)$. The rank of this homogeneous system equals $r$, where $r=\cod (E\cap P(W))-\cod E$, and thus, among all $s$ equalities of $P(W)$, we can choose just $r$ equalities to obtain $E\cap X$. Hence, there is a polydiagonal $X$ such that $E\cap X=E\cap P(W)=W$ and $$\cod X=r= \cod (E\cap P(W))-\cod E.$$ This implies that $\cod X=\dim E- \dim W$, concluding the proof.\end{proof}

\begin{ex}\normalfont \label{CalcSS} Consider the following 2-dimensional subspace of $\C^5$:
$$E=\{x_2=x_3,\, x_4=x_5, \, x_1+x_2+x_4=0\}.$$
All 1-dimensional special subspaces in $E$ are obtained intersecting $E$ with all 1-dimensional polydiagonals defined by a unique equality of coordinates independent from the equalities defining $E$. Hence, it suffices to consider three polydiagonals, for example, $\{x_1=x_2\}$, $\{x_1=x_4\}$ and $\{x_2=x_4\}$. Therefore, there are exactly three 1-dimensional special subspaces in $E$:
\begin{enumerate}
\item $  \, W_{1}= \{x_1=x_2=x_3, x_4=x_5,  2x_1+x_4=0\}$,
\item $ \, W_{2}=\{x_1=x_4=x_5, x_2=x_3,  2x_1+x_2=0\}$,
\item $ \, W_{3}=\{x_2=x_3= x_4=x_5, x_1+2x_2=0\}$.\end{enumerate}\end{ex}

The following result is crucial to prove Theorem \ref{FT}, which is the main result of this work about special Jordan subspaces.

\begin{prop} \label{DSDsubspaces} Consider a $k$-dimensional subspace $E$ of $\C^n$, with $k>1$. If $E$ does not contain the fully synchrony subspace then there is a direct sum decomposition of $E$ into 1-dimensional special subspaces in $E$.
\end{prop}

\begin{proof} The $n-1$ equalities of the fully synchrony subspace $F$ in the subspace $E$ define an homogeneous linear system with $n-1$ equations. The rank of this homogeneous system equals $r$, where $r=\cod (E\cap F)-\cod E=n-\cod E=\dim E=k$. Thus, we can choose $k$ equalities of coordinates to obtain $E\cap F=\{0\}$. So, there is a polydiagonal $X$ satisfying $$E\cap X=E\cap F=\{0\}, \quad \cod X=k.$$

Let $X_1,\ldots,X_{k}$ be the $k$ possible distinct codimension-$(k-1)$ polydiagonals that contain $X$. Note that for every $1\le i\le k$, the intersection $E\cap X_i$ is a 1-dimensional special subspace in $E$ and that $$i\ne j\quad \Rightarrow \quad E \cap X_i\cap X_j=E\cap X=\{0\}.$$
Therefore, $E=(E\cap X_1) \oplus \cdots \oplus (E\cap X_{k})$ is a direct sum decomposition of $E$ into 1-dimensional special subspaces in $E$. \end{proof}

\section{Special Jordan Subspaces}


In this section we consider regular networks and we define special Jordan subspaces to these networks. 

\subsection{Jordan subspaces} \label{basic} We start providing basic concepts and results about Jordan subspaces that are extremely important in our work. For details about this subject see, for example, Section 6 of  Lancaster and Tismenetsky \cite{LT}.

Consider a linear transformation $A$ from $\C^n$ into $\C^n$ and an eigenvalue $\lambda$. For a positive integer $r$, the subspace $K_{\lambda}^r = \Ker (A-\lambda I)^r$ is the {\em generalized eigenspace of 
order} $r$ (to $\lambda$). Since $A$ has finite order, there is a positive integer $p$ such that $$ K_{\lambda}^1 \subset \cdots \subset  K_{\lambda}^p = K_{\lambda}^{p+1} = .... $$ For such a $p$, the generalized eigenspace  $K_{\lambda}^p$ is the {\em generalized eigenspace} to $\lambda$ and it is simply denoted by $G_\lambda$. An element $x$ in $K_{\lambda}^r \backslash \ K_{\lambda}^{r-1}$ is called a {\em generalized eigenvector} of order $r$ of $A$, considering $K_{\lambda}^0=\{ 0 \}$. An ordinary eigenvector is a generalized eigenvector of order 1 and an eigenspace is a generalized eigenspace of order 1.

A sequence of nonzero vectors $\{x_1,\ldots, x_k \}$ is called a {\em Jordan chain} of length $k$ associated with the eigenvalue $\lambda$ when $$(A-\lambda I) x_1 = 0, \, (A-\lambda I) x_2 =x_1,\cdots, (A-\lambda I)x_k=x_{k-1}.$$ Any Jordan chain consists of linearly independent vectors (\cite{LT}, p. 230) and so, if it has length $k$, it spans a $k$-dimensional subspace. A {\em Jordan subspace} is a subspace spanned by a Jordan chain. 

\subsection{Special Jordan subspaces} In this subsection we define {\em special Jordan subspaces} and we provide a method to obtain these subspaces. 

\begin{dfn} \normalfont \label{sjs}Consider a regular network and the corresponding adjacency matrix. A Jordan subspace $W$ of a generalized eigenspace $G$ is {\em special to the network} when, for every Jordan subspace $U$ of $G$, $$[\dim U=\dim W,  \;  P(U)\subset P(W)]\, \Rightarrow \, [ P(U)=P(W)  \, \vee\, U=F  ],$$ where $F$ is the fully synchrony subspace.\end{dfn}

Thus, $W$ is a special Jordan subspace to the network when every Jordan subspace $U$ of $G$ having the same dimension and satisfying the same equalities of coordinates as $W$, either does not satisfy any additional equality of coordinates or is the fully synchrony subspace. 

After the introduction of this definition it is necessary to take extra care with the word {\em special}. The terms {\em special subspace} and {\em special Jordan subspace} must not be confused. In fact, the first concept is defined for all subspaces of $\C^n$; the second concept is only defined for Jordan subspaces in the context of regular networks. In general, these concepts are distinct. Trivially:
\begin{enumerate}
\item for regular networks with diagonalizable adjacency matrices, all Jordan subspaces are 1-dimensional.
\item if $G$ is an eigenspace that is not associated with the valency of the network, 1-dimensional special Jordan subspaces to the network are precisely 1-dimensional special subspaces in $G$.
\item if $G$ is the eigenspace associated with the valency of the network (recall that the valency is always a semi-simple eigenvalue) then: \begin{enumerate}
\item $F$ is the unique 1-dimensional special subspace in $G$,
\item all special Jordan subspaces to the network in $G$ are 1-dimensional,
\item if $\dim G=2$ then all 1-dimensional subspaces are special Jordan subspaces.
\end{enumerate}
\item for generalized eigenspaces of order $k>1$, not all $l$-dimensional special subspaces are $l$-dimensional special Jordan subspaces, with $1\le l\le k$.
\end{enumerate}

\subsection{Special Jordan subspaces in eigenspaces that are not associated with the valency} As said above, if $G$ is an eigenspace (generalized eigenspace of order 1) that is not associated with the valency of the network then a subspace $J$ of $G$ is a 1-dimensional special subspace in $G$ if and only if it is a special Jordan subspace to the network. So, the method to list all special Jordan subspaces in $G$ consists solely in calculating all 1-dimensional special subspaces in $G$.

\begin{ex} \normalfont \label{Ex+-i}Consider again the 5-cell regular network in Figure \ref{Fig+-i} and the eigenspaces of the corresponding adjacency matrix:
$$\begin{array}{l}G_2= \{x_1=x_2=x_3=x_4=x_5\}, \\ G_{-1}= \{x_2=x_3, \, x_4=x_5, x_1+x_2+x_4=0\},\\
G_{\pm i}= \{x_1=x_4,  2x_1+2x_2+x_3=0, \, \pm (2 i) x_1+x_3=0, 3x_1+x_2+x_5=0\}.\end{array}$$

$G_2$ and $G_{\pm i}$ are 1-dimensional and, therefore, they are special Jordan subspaces to the network. For the eigenspace $G_{-1}$, all special Jordan subspaces  in this eigenspace are precisely all 1-dimensional special subspaces in $G_{-1}$, which were already obtained in Example \ref{CalcSS}:
\begin{enumerate}
\item $  \, W_{-1,1}= \{x_1=x_2=x_3, x_4=x_5,  2x_1+x_4=0\}$,
\item $ \, W_{-1,2}=\{x_1=x_4=x_5, x_2=x_3,  2x_1+x_2=0\}$,
\item $ \, W_{-1,3}=\{x_2=x_3= x_4=x_5, x_1+2x_2=0\}$.\end{enumerate}

Thus, there are only six special Jordan subspaces to this network. 
\end{ex}

We use this example to point out two important facts about special Jordan subspaces:
\begin{enumerate}
\item The definition of special Jordan subspaces strongly depends on the corresponding eigenspace. Indeed, $G_{-1}$ and $W_{-1,3}$ are 1-dimensional special Jordan subspaces and $P(W_{-1,3})\subsetneq P(G_{-1})$. This fact is only possible because the subspaces belong to distinct eigenspaces.
\item The number of equalities of coordinates satisfied by all special Jordan subspaces with a prescribed dimension is not fixed, in the sense that two special Jordan subspaces with the same dimension may satisfy a different number of equalities of coordinates.
\end{enumerate}

\subsection{Special Jordan subspaces in the eigenspace associated with the valency} Given a regular network, consider the eigenspace $G$ associated with the valency of the network. If $G$ is 1-dimensional then $G$ is the fully synchrony subspace. If $G$ has a higher dimension, there is an infinite number of special Jordan subspaces in $G$. Nevertheless, we prove that the set of the smallest polydiagonals containing these special Jordan subspaces is finite.

\begin{prop} \label{valency} Given a regular network, let $G$ be the eigenspace associated with the valency, with $\dim G>1$, and let $E$ be a direct complement to the fully synchrony subspace $F$ in $G$.\begin{enumerate}
\item For every special Jordan subspace $J\ne F$ in $G$, there is a special subspace $J'$ in $E$ such that $P(J)=P(J')$ and $F\oplus J=F\oplus J'$.
\item  \label{valencyspecial}  A subspace of $E$ is a special Jordan subspace to the network if and only if it is a special subspace in $E$.
\item \label{CorValency} If $S=F\oplus J_1\oplus \cdots \oplus J_k$ is a direct sum of 1-dimensional special Jordan subspaces in $G$ then there are special subspaces $J_1',\cdots, J_k'$ in $E$ such that $S= F\oplus J_1'\oplus \cdots \oplus J_k'$.\end{enumerate}
\end{prop}

\begin{proof} (1) Let $J$ be a special Jordan subspace to the network in $G$, $J\ne F$. Notice that $J$ is 1-dimensional. If $J$ is a subspace of $E$ then it is clear from Definition \ref{sjs} that $J$ is special in $E$. So, assume that $J$ is not a subspace of $E$. Then, $$J=\mbox{span}\{w\}, \mbox{ for some }w\in G,\,  w\notin E, \,w\notin F.$$ Because $G=F\oplus E$, we have $w=f+u$, for some $f\in F$ and $u\in E\backslash \{ 0\}$. Thus, $$J \subset F\oplus \mbox{span}\{u\}  \; \mbox{ and } \; P(J)=P(\mbox{span}\{u\}).$$ Therefore, $J'=\mbox{span}\{u\}$ is also a special Jordan subspace to the network in $E$ and so, it is a special subspace in $E$. Moreover, $$J \subset F\oplus J'\; \Rightarrow\; F\oplus J\subset F\oplus J'$$ and, because $\dim (F\oplus J) = \dim (F\oplus J')$, we conclude that $F\oplus J = F\oplus J'$.

(2) It is straightforward that every special Jordan subspace is a special subspace in the corresponding eigenspace. Reciprocally, let $W$ be a special subspace in $E$ and $J$ be a special Jordan subspace of $G$ such that $P(J)\subset P(W)$. Due to (1), there is a special subspace $J'$ in $E$ such that $P(J)=P(J')$, and so, $P(J')\subset P(W)$. Therefore, since $W$ is special in $E$, $P(W)=P(J')$ and $W$ is a special Jordan subspace to the network.

(3) If $S=F\oplus J_1\oplus \cdots \oplus J_k$ is a direct sum of 1-dimensional special Jordan subspaces in $G$ then, due to (1), there are special subspaces $J_1',\cdots, J_k'$ in $E$ such that $F\oplus J_i=F\oplus J_i'$,  for all $1\le i\le k$. Thus, $S=F+J_1'+ \cdots + J_k'$. Besides, $\dim (F+J_1' + \cdots + J_k')=\dim S=k+1$ and so, $S= F\oplus J_1'\oplus \cdots \oplus J_k'$.\end{proof}

This result shows that, despite the infinite number of special Jordan subspaces in $G$ when $\dim G>1$, the set of the smallest polydiagonals containing these subspaces is finite and it consists of the fully synchrony subspace $F$ and the smallest polydiagonals containing all special subspaces in an arbitrary direct complement $E$ to $F$ in $G$. Moreover, a direct sum of $F$ with other special Jordan subspaces of $G$ can be reduced to a direct sum of $F$ with special Jordan subspaces of $E$. 

\begin{ex} \normalfont \label{ExValMult} Consider the following two regular networks in Figure \ref{FigValMult}.
\begin{figure}[!htb] \includegraphics[scale=0.37]{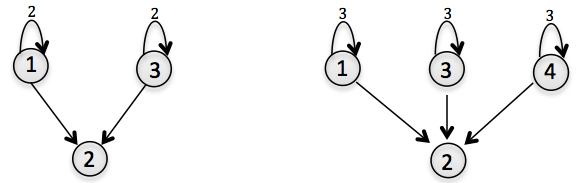}\caption{Regular networks for which the valency is a multiple eigenvalue.}\label{FigValMult}\end{figure}

For the 3-cell network, the valency is an eigenvalue with multiplicity 2 (algebraic and geometric) and the corresponding eigenspace is $$G_2=\{x_1-2x_2+x_3=0\}.$$ As remarked above, all 1-dimensional subspaces in this eigenspace are special Jordan subspaces to the network. According to Proposition \ref{valency}, for all subspaces $J_1$ distinct from the fully synchrony subspace, $P(J_1)$ is the total phase space.

For the 4-cell network, 
the valency is an eigenvalue with multiplicity 3 and the corresponding eigenspace is
$$G_3=\{x_1-3x_2+x_3+x_4=0\}.$$ There is an infinite number of special Jordan subspaces $J_2$ to this network in $G_3$. However, using to Proposition \ref{valency} and choosing the direct complement $$E=\{x_1+x_3+x_4=0, \, x_2=0\},$$ $P(J_2)=P(J')$, where $J'$ is one of the following special subspaces in $E$:
$$ \begin{array}{ll}
\{x_1=x_2=0,\, x_3+x_4=0\}, & \{x_1=x_3,\, x_2=0, \, 2x_1+x_4=0\},\\
\{x_2=x_3=0,\, x_1+x_4=0\}, & \{x_1=x_4,\, x_2=0, \, 2x_1+x_3=0\},\\
\{x_2=x_4=0,\, x_1+x_3=0\}, & \{x_3=x_4, \, x_2=0,\, x_1+2x_3=0\}.\end{array}$$
\end{ex}


\subsection{Special Jordan subspaces in generalized eigenspaces of order greater than 1} Given a regular network, consider a generalized eigenspace $G_{\lambda}$ of order greater than 1. In this subsection we guarantee the existence of some special Jordan subspaces.

The method to obtain all $k$-dimensional special Jordan subspaces in $G_\lambda$ can be done, for example, taking into account that every subspace is contained in a $k$-dimensional special subspace in $\Ker (A-\lambda I)^k$, where $A$ is the adjacency matrix of the network. Furthermore, every subspace must contain a $(k-1)$-dimensional Jordan subspace and thus the calculations can be simplified by demanding the corresponding inclusion between the smallest polydiagonals containing each of these two subspaces.

\begin{ex} \normalfont \label{ExND1}Consider the 5-cell regular network in Figure \ref{FigND1} and the generalized eigenspaces of the corresponding adjacency matrix:
$$\begin{array}{l} G_2= \{x_1=x_2=x_3=x_4=x_5\}, \\
                   G_1= \{x_2=x_5, \; x_3=x_4, \; x_1=0, \; x_2+x_3=0\},  \\ 
                   G_{-1}= \{3x_1+4x_2-x_3+3x_5=0, \; x_2-x_3-x_4+x_5=0\}.
                   \end{array}$$
                   
\begin{figure}[!hb]
\begin{minipage}[b]{0.37\linewidth}\centering
        \includegraphics[scale=0.37]{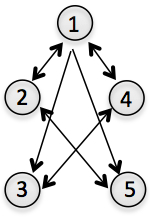} \vspace{-1.1cm}
 \end{minipage}  \begin{minipage}[b]{0.4\linewidth} \centering
$$A= \left(\begin{array}{ccccc}
0 & 1 & 0 & 1 & 0  \\
1 & 0 & 0 & 0 & 1  \\
1 & 0 & 0 & 1 & 0  \\
1 & 0 & 1 & 0 & 0  \\
1 & 1 & 0 & 0 & 0    \end{array}\right)$$
 \end{minipage}
 
 \caption{Regular network and the corresponding adjacency matrix.} \label{FigND1}
 \end{figure}

$G_2$ and $G_1$ are 1-dimensional and, therefore, they are special Jordan subspaces to the network. $G_{-1}$ is a generalized eigenspace of order 2 and
$$K^1=\Ker (A+I)=\{x_2=x_3, \, x_4=x_5, \, x_1+x_2+x_4=0\}.$$ 
All 1-dimensional special Jordan subspaces in $G_{-1}$ are precisely all 1-dimensional special subspaces in $K^1$ and they were already obtained in Example \ref{CalcSS}:
\begin{enumerate}
\item $  \, W_{-1,1}= \{x_1=x_2=x_3, x_4=x_5,  2x_1+x_4=0\}$,
\item $ \, W_{-1,2}=\{x_1=x_4=x_5, x_2=x_3,  2x_1+x_2=0\}$,
\item $ \, W_{-1,3}=\{x_2=x_3= x_4=x_5, x_1+2x_2=0\}$.\end{enumerate}

To obtain all 2-dimensional special Jordan subspaces, \vspace{0.2cm}notice that:

$K^1\cap \Im (A+I)  =\{x_2=x_3, x_4=x_5, x_1+x_2+x_4=0\}\cap \{x_2=x_5,x_3=x_4\}$\\
\hspace*{3cm} $=W_{-1,3}.\vspace{0.2cm}$\\
So, all 2-dimensional special Jordan subspaces contain $W_{-1,3}$.
Therefore, if $J_2$ is a special Jordan subspace in $G_{-1}$ then $P(J_2)$ contains $P(W_{-1,3})$. There are exactly four 2-dimensional special subspaces in $G_{-1}$ containing $W_{-1,3}$, namely:
$$\begin{array}{l}K^1= \{x_2=x_3, \, x_4=x_5, \, x_1+x_2+x_4=0\}, \\
			    J^2_{-1,1} \, = \, \{ x_2=x_4, \; x_3=x_5,\; 3x_1+4x_2+2x_3=0\}, \\
                                J^2_{-1,2} \, = \, \{ x_2=x_5, \; 3x_1+7x_2-x_3=0, \; 2x_2-x_3-x_4=0\},\\
                                J^2_{-1,3}\, = \, \{ x_3=x_4, \; 3x_1+4x_2-x_3+3x_5=0, \; x_2-2x_3+x_5=0\}.\end{array}$$

It is easy to see that the last three subspaces are indeed 2-dimensional Jordan subspaces because they are invariant. Besides, no 2-dimensional Jordan subspace satisfies either $x_2=x_3$ or $x_4=x_5$ because $$G_{-1}\cap \{x_2=x_3\}=G_{-1}\cap \{x_4=x_5\}=K^1.$$ Thus, we conclude that there are exactly eight special Jordan subspaces to this network: 
$$G_2, \, G_{1}, \, W_{-1,1}, \, W_{-1,2}, W_{-1,3}, J^2_{-1,1}, \, J^2_{-1,2}, \, J^2_{-1,3}.$$\end{ex}

The next result guarantees the existence of some special Jordan subspaces. 

\begin{teo} \label{quebracabecas} Given a regular network with adjacency matrix $A$, consider an $l$-dimensional generalized eigenspace $G_\lambda$, with $l\ge 1$. For a positive integer $j$, set $$N^j=\Ker (A-\lambda I)\cap \Im (A-\lambda I)^{j-1}.$$ For every $1\le k\le l$ and for every 1-dimensional special subspace in $N^k$,  there is a $k$-dimensional special Jordan subspace containing it.\end{teo}

\begin{proof} The result is trivial when $l=1$ and thus we assume $l>1$. Consider a positive integer $k\le l$ and a $1$-dimensional special subspace $J_1$ in $N^k$. By Proposition \ref{specialvector}, consider a polydiagonal $Q$ satisfying:
$$J_1=Q\cap N^k, \; P(J_1)\subset Q, \; \cod Q=\dim N^k-1.$$

Consider the subspace
$$J_{2}'=Q\cap (A-\lambda I)^{-1}(J_1)\cap \Im (A-\lambda I)^{k-2}.$$ All equalities defining $Q$ are independent from all equalities defining $N^k$ and so, they are also independent from all equalities defining any subspace containing $N^k$. Hence:

$\begin{array}{ll} (1)\;  \cod J_{2}' & =\cod Q+\cod [(A-\lambda I)^{-1}(J_1)\cap \Im (A-\lambda I)^{k-2}]\\  & =\dim N^k-1+n-(\dim N^{k-1}+1)\\ & =n-(\dim N^{k-1}-\dim N^k +2),\end{array}$

$\begin{array}{ll} (2) \; \cod (J_{2}'\cap K^1)&=\cod (J_{2}'\cap N^{k-1})\\  & =n-(\dim N^{k-1}-\dim N^k+1).\end{array}$

and thus,\begin{enumerate} \item $\dim J_{2}'=\dim N^{k-1}-\dim N^k +2$, \item $\dim (J_{2}'\cap K^1)=\dim N^{k-1}-\dim N^k +1$. \end{enumerate}Consequently, there is a system $B_1$ of linearly independent vectors of $N^{k-1}$ such that $N^{k-1}=\mbox{span}(B_1)\oplus N^k$ and $J_{2}'=\mbox{span}(B_1) \oplus J_1\oplus \mbox{span}\{x_{2}\}$, with $$x_{2}\in (A-\lambda I)^{-1}(J_1)\cap \Im (A-\lambda I)^{k-2} \cap (K^2\backslash K^1).$$

 So, $J_{2}=J_1\oplus \mbox{span} \{x_{2}\}$ is a Jordan subspace containing $J_1$ and satisfying $P(J_{2})\cap N^k =J_1$. This Jordan subspace $J_{2}$ is contained in $\Im (A-\lambda I)^{k-2}$ and thus we can consider the subspace $$J_{3}'=Q_j\cap (A-\lambda I)^{-1}(J_{2})\cap \Im (A-\lambda I)^{k-3}.$$ Analogously, there is a system $B_2$ of linearly independent vectors of $N^{k-2}$ such that $N^{k-2}=\mbox{span}(B_2)\oplus N^k$ and $J_{3}'=\mbox{span}(B_2)\oplus J_{2}\oplus \mbox{span}\{x_{3}\}$, with $$x_{3}\in (A-\lambda I)^{-1}(J_{2})\cap \Im (A-\lambda I)^{k-3} \cap (K^3\backslash K^2).$$
 
 So, $J_{j3}=J_{2}\oplus \mbox{span}\{x_{3}\}$ is a Jordan subspace containing $J_1$ and satisfying $P(J_{3})\cap N^k=J_1$. Applying the same procedure, we obtain a 4-dimensional Jordan subspace $J_{4}$ containing $J_1$ and satisfying $P(J_{4}) \cap N^k=J_1$. Continuing this process successively, we obtain a $k$-dimensional Jordan subspace $J_{k}$ containing $J_1$ and satisfying $P(J_{k})\cap N^k=J_1$.

It remains to prove that there is a special Jordan subspace containing $J_1$. Let $Y_{k}$ be a $k$-dimensional special Jordan subspace to the network such that $P(Y_{k})\subset P(J_{k})$. Then,
$$Y_{k}\cap N^k \subset P(Y_{k})\cap N^k\subset P(J_{k})\cap N^k =J_1,$$
and so, the 1-dimensional subspace $Y_{k}\cap N^k$ is $J_1$. Therefore, $Y_{k}$ is a special Jordan subspace containing $J_1$.\end{proof}

\subsection{Main theorem on special Jordan subspaces} In this subsection we prove the main result of this work on special Jordan subspaces.

\begin{teo} Given an $n$-cell regular network, there exists a direct sum decomposition of $\C^n$ into special Jordan subspaces to the network.\label{FT}\end{teo}

\begin{proof} $C^n$ is the direct sum of all generalized eigenspaces (\cite{GLR}, p. 47) and thus it suffices to prove that every generalized eigenspace $G_\lambda$ admits a direct sum decomposition into special Jordan subspaces to the network.

If $G_\lambda$ is an eigenspace that is not associated with the valency of the network, the result follows from Proposition \ref{DSDsubspaces}.

If $G_\lambda$ is the eigenspace associated with the valency, the result is obvious when $G_\lambda$ is the fully synchrony subspace $F$. Otherwise, there is a nonzero direct complement $E$ to $F$ in $G_\lambda$ and, by Proposition \ref{DSDsubspaces}, there is a direct sum decomposition of $E$, $E=J_1\oplus \cdots \oplus J_k$, into special subspaces in $E$, which are special Jordan subspaces to the network due to Proposition \ref{valency}. Hence, $G_\lambda=F\oplus J_1\oplus \cdots \oplus J_k$ is a direct sum decomposition of $G_\lambda$ into special Jordan subspaces to the network.

If $G_\lambda$ is a generalized eigenspace of order $\mu >1$, let $$N^i=\Ker (A-\lambda I)\cap \Im (A-\lambda I)^{i-1} \, \mbox{ and }\, \nu_i=\dim N^i,$$ with $1\le i\le  \mu$. 
Firstly, by Proposition \ref{DSDsubspaces}, consider a direct sum decomposition of $N^{\mu}$ into 1-dimensional special subspaces of $N^\mu$. Since $N^\mu\subset N^{\mu-1}$, we can add $(\nu_{\mu-1}-\nu_\mu)$ 1-dimensional special subspaces in $N^{\mu-1}$ to the previous sum and obtain a direct sum decomposition of $N^{\mu-1}$. Continuing this process, we obtain a direct sum decomposition of $N^1$, $N^1=W_1\oplus\cdots\oplus W_{\nu_1}$, that includes, for each $1\le i\le \mu$, exactly $(\nu_{i}-\nu_{i+1})$ 1-dimensional special subspaces of $N^{i}$, considering $\nu_{\mu+1}=0$. Finally, by Theorem \ref{quebracabecas}, for each special subspace $W_j$ in $N^i\backslash N^{i-1}$, $1\le i\le \mu$, we can consider an $i$-dimensional special Jordan subspace $J_j$ containing $W_j$, considering $N^0=\{0\}$. The set of the constructed special Jordan subspaces is linearly independent (\cite{LT}, p. 233) and so, $G_\lambda=J_1\oplus \cdots\oplus J_{\nu_1}$ is a direct sum decomposition of the whole generalized eigenspace $G_\lambda$.
\end{proof}

\section{List of all synchrony subspaces} In this section we relate special Jordan subspaces with synchrony.

\subsection{Main theorem on synchrony} In this subsection we prove the main result of this work on synchrony.

\begin{teo} Given a regular network, a subspace of the total phase space is a synchrony subspace if and only if it is polydiagonal and a direct sum of special Jordan subspaces to the network.\label{FT2}\end{teo}

\begin{proof}  By Lemma \ref{known}, it suffices to prove that every synchrony subspace admits a direct sum decomposition into special Jordan subspaces. Let $G$ be an $n$-cell regular network, $S$ a synchrony subspace, $Q$ the corresponding $m$-cell quotient network and let $A$ and $A_Q$ be the adjacency matrices of $G$ and $Q$, respectively. By Theorem \ref{FT}, there is a direct sum decomposition of $C^m$, $C^m=J_1\oplus \cdots \oplus J_m$, into special Jordan subspaces to $Q$.

Consider the isomorphism $\phi$ from $C^m$ into $S$ defined by the natural identification between these two spaces and notice that for every real $\mu$, $$A-\mu I_n=\phi \circ (A_Q-\mu I_m)\circ \phi^{-1}.$$ 

$S$ can be written as: $$S=\phi(C^m)=\phi(J_1\oplus \cdots \oplus J_s)=\phi(J_1)\oplus \cdots \oplus \phi(J_m).$$ Notice that if $i\ne j$, $\phi(J_i)\cap \phi(J_j)=\{0\}$. Next we prove that the last equality represents a direct sum decomposition of $S$ into special Jordan subspaces. For each $i=1,\cdots, m$:
\begin{enumerate}
\item $\phi (J_i)$ is a Jordan subspace to $G$: assuming that $\{x_1,\cdots,x_k\}$ is a Jordan chain spanning $J_i$ in a generalized eigenspace $G_{\lambda_i}$ then, for $1\le j\le k$, $$\hspace*{1cm}\phi(x_j)\ne 0, \;(A-\lambda_i I_n)(\phi(x_j))=  \phi ((A_Q-\lambda_i I_m)(x_j)) = \phi (x_{j-1}),$$ considering $x_0=0$. So, $\{\phi(x_1),\cdots,\phi(x_k)\}$ is a Jordan chain in $S$ spanning $\phi(J_i)$.

\item $\phi(J_i)$ is a special Jordan subspace to $G$: let $\tilde X_i$ be a special Jordan subspace to $G$ such that $P(\tilde X_i)\subset P(\phi(J_i))$. $\tilde X_i$ is contained in $S$ because $\tilde X_i\subset P(\tilde X_i)\subset P(\phi(J_i))\subset S$ and so, we can consider the pre-image $X_i=\phi^{-1}(\tilde X_i)$. Analogously to what was done in (1), we prove that $X_i$ is a Jordan subspace of $A_Q$. Therefore,
$$P(\tilde X_i)\subset P(\phi(J_i)) \Rightarrow P(X_i)\subset P(J_i) \Rightarrow X_i=J_i \Rightarrow \tilde X_i=\phi(J_i),$$
and thus $\phi (J_i)$ is a special Jordan subspace to $G$.\end{enumerate}
Hence, $S=\phi(J_1)\oplus \cdots \oplus \phi(J_m)$ is a direct decomposition of $S$ into special Jordan subspaces to $G$.\end{proof}

\begin{cor} Given a regular network, there are 2-dimensional synchrony subspaces if and only if there are eigenvectors with exactly two distinct coordinates.\label{bi}\end{cor}

\begin{proof} A 2-dimensional synchrony subspaces is precisely a direct sum of the fully synchrony subspace with a 1-dimensional special Jordan subspaces satisfying $n-2$ equalities of coordinates, where $n$ is the number of cells of the network.\end{proof}

\begin{cor} Given an $n$-cell regular network, a polydiagonal defined by $n-l$ equalities of coordinates is a synchrony subspace if and only if these equalities are satisfied by $l$ special Jordan subspaces whose sum is $l$-dimensional.\label{eqcoord}
\end{cor}

\begin{proof} Firstly, assume that $S$ is a synchrony subspace defined by $n-l$ equalities. Using Theorem \ref{FT2}, $S$ is the direct sum of special Jordan subspaces whose sum is $l$-dimensional. Hence, because each of them is contained in $S$, each of them also satisfies the $n-l$ equalities of $S$.

Conversely, assume that there are $l$ special Jordan subspaces satisfying $n-l$ equalities of coordinates and whose sum is $l$-dimensional. Then, their direct sum has codimension $n-l$ and is defined by the common $n-l$ equalities of coordinates. Thus, it is polydiagonal. Besides, it also invariant because it is the sum of invariant subspaces. So, the direct sum is a synchrony subspace.\end{proof}

\subsection{Method to list all synchrony subspaces}

Theorem \ref{FT2} and Corollary \ref{eqcoord} provide a useful method to list all synchrony subspaces of an $n$-cell regular network:
\begin{enumerate}
\item Compute all special Jordan subspaces (all possible dimensions);
\item For each $l\ge 1$, analyze if there are $l$ special Jordan subspaces with $n-l$ common equalities of coordinates and whose sum is $l$-dimensional. Only in the affirmative case this sum is a synchrony subspace.\end{enumerate}

\begin{remark} \normalfont Note that Step 2 can be easily done just by looking at the list of all special Jordan subspaces. \end{remark}

\begin{ex}\normalfont In Example \ref{Ex+-i}, we calculated all special Jordan subspaces of the 5-cell regular network in Figure \ref{Fig+-i}, namely:

$G_2=\{x_1=x_2=x_3=x_4=x_5\}$,
 
$G_{\pm i}= \{x_1=x_4,  2x_1+2x_2+x_3=0,  \pm (2 i) x_1+x_3=0, 3x_1+x_2+x_5=0\}$,
 
$ W_{-1,1}= \{x_1=x_2=x_3, x_4=x_5,  2x_1+x_4=0\}$,

$ W_{-1,2}=\{x_1=x_4=x_5, x_2=x_3,  2x_1+x_2=0\}$,
 
$W_{-1,3}=\{x_2=x_3= x_4=x_5, x_1+2x_2=0\}$.

For each $1\le l\le 4$, we analyze if there are $l$ special Jordan subspaces with $5-l$ common equalities of coordinates  and whose sum is $l$-dimensional, obtaining exactly five nontrivial synchrony subspaces:
\begin{enumerate}
\item $G_2 \oplus W_{-1,1}= \{x_1=x_2=x_3, \, x_4=x_5\}$,
\item $G_2 \oplus W_{-1,2}= \{x_1=x_4=x_5, \, x_2=x_3\}$,
\item $G_2  \oplus W_{-1,3}= \{x_2=x_3=x_4=x_5\}$,
\item $G_2 \oplus W_{-1,1} \oplus W_{-1,2}= \{x_2=x_3, \, x_4=x_5\}$,
\item $G_2 \oplus G_{i}\oplus G_{-i} \oplus W_{-1,2}= \{x_1=x_4\}$.
\end{enumerate}
The corresponding lattice is presented in Figure \ref{Lattice+-i}.
\end{ex}

\begin{figure}[!ht]\centering \includegraphics[scale=0.3]{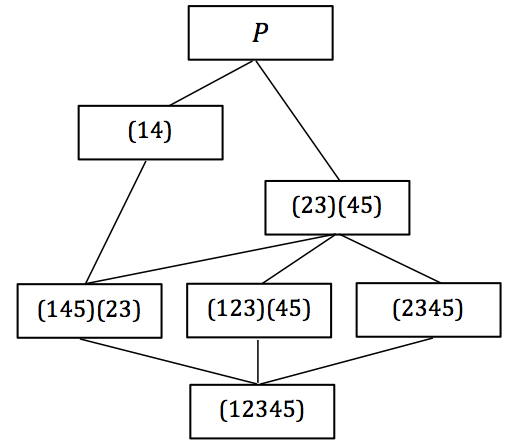} \caption{The lattice of all synchrony subspaces of the network in Figure \ref{Fig+-i}. Each cycle $(i_1,\cdots, i_s)$ denotes the equality of the corresponding cell coordinates and $P$ denotes the total phase space.}\label{Lattice+-i}\end{figure}


\section{Special Jordan subspaces and lattices of synchrony subspaces}

In this section we consider lattices of synchrony subspaces of regular networks and we show that special Jordan subspaces play a special role in the structure of these lattices.

\subsection{The lattice of all synchrony subspaces}

Given a linear transformation $A$ from $\C^n$ into $C^n$, the set of all $A$-invariant subspaces is a lattice, where the meet and the join operations are the intersection and sum of sets, respectively (\cite{GLR}, p. 31). 
Stewart proved that the set of all synchrony subspaces of a given network is a complete lattice and that this lattice is not a sublattice of the lattice of all $A$-invariant subspaces, where $A$ is the adjacency matrix of the network \cite{S09}. In fact, although the meet operation is the same for both lattices, the join is not because the sum of two synchrony subspaces is not always a synchrony subspace. In fact, it is straightforward that the sum of two synchrony subspaces is a synchrony subspace if and only if the sum is polydiagonal.

\begin{ex} \normalfont Using this result, it is very easy to understand that the lattice structure $L_{14}$ on Figure \ref{L14fig} presented by Kamei in \cite{K09}, for 4-cell regular networks whose adjacency matrices have only simple eigenvalues, can be eliminated. This statement answers the query in Remark 6.1 of \cite{K09}.

\begin{figure}[!htb]\centering \includegraphics[scale=0.11]{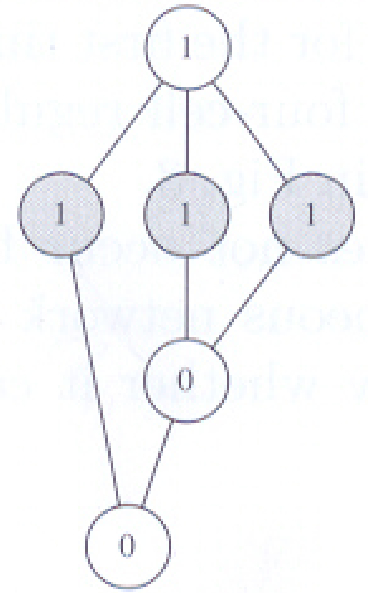}\caption{Lattice structure $L_{14}$ of Kamei \cite{K09}.}\label{L14fig}\end{figure}

In fact, in that work, the $l$-th level of a lattice includes all codimension-$(l-1)$ synchrony subspaces. Hence, as we see in Figure \ref{L14fig}, structure $L_{14}$ is relative to a network with exactly three  codimension-2 synchrony subspaces and so they are exactly three of the following polydiagonal subspaces:
$$\begin{array}{l}S_1=\{x_1=x_2=x_3\}, \\ S_2=\{x_1=x_2=x_4\}, \\S_3=\{x_1=x_3=x_4\}, \\   S_4=\{x_2=x_3=x_4\},\end{array} \begin{array}{l}  T_1=\{x_1=x_2,\, x_3=x_4\}, \\ T_2=\{x_1=x_3,\, x_2=x_4\},\\ T_3=\{x_1=x_4,\, x_2=x_3\}.  \end{array}$$

The figure shows that the sum of two of them is a codimension-1 polydiagonal subspace. Therefore, at least one of them is $S_i$, for some $1\le i\le 4$. But the sum of this synchrony subspace with the third codimension-2 synchrony subspace is  also a codimension-1 polydiagonal, and so, it is a synchrony subspace, fact that contradicts $L_{14}$. Hence, this structure can be eliminated.

\end{ex}

\subsection{Special Jordan subspaces and the lattice structure}

Theorem \ref{FT2} implies that all special Jordan subspaces are sufficient to list all synchrony subspaces and so, to construct the corresponding lattice. Moreover, special Jordan subspaces play a special role in the structure of these lattices because they are directly connected with the join-irreducible elements of such lattices (elements that can not be obtained as the join of two other elements of the corresponding lattice).

\begin{prop} Given a regular network, consider the lattice of all synchrony subspaces. Every join-irreducible element of this lattice is the smallest synchrony subspace containing some special Jordan subspace. Moreover, the number of join-irreducible elements of this lattice does not exceed the number of special Jordan subspaces.
\end{prop}

\begin{proof} Let $S$ be a join-irreducible element of the lattice. The result is obvious when $S$ is the fully synchrony subspace and so we suppose that $S$ is another synchrony subspace. By Theorem \ref{FT2}, $S$ can be written as a direct sum of special Jordan subspaces, say $J_1,\cdots,J_k$. Let $S_1,\cdots, S_k$ be the smallest synchrony subspaces containing $J_1,\cdots, J_k$, respectively. Then, $S_1,\cdots,S_k$ are elements of the considered lattice and $S=S_1+\cdots+S_k$. Since $S$ is a join-irreducible element, $S_i=S$, for some $1\le i\le k$. Thus, $S$ is the smallest synchrony subspace containing at least one of the special Jordan subspaces. As an immediate consequence, we get the last statement of this proposition.\end{proof}

\begin{ex}\normalfont  In Examples \ref{Ex+-i} and \ref{ExND1} the number of special Jordan subspaces and join-irreducible subspaces coincides in each case.
\end{ex}

Our study shows that a higher number of special Jordan subspaces implies, in general, a higher number of synchrony subspaces because this fact increases the chances of appearing common equalities of coordinates and thus, increases the possibility of appearing sums of special Jordan subspaces that are polydiagonal. Besides, the higher dimension of an eigenspace implies the higher number of special Jordan subspaces. Therefore, among networks with the same number of cells, it is natural to expect a higher number of synchrony subspaces and of possible lattices structures:
\begin{enumerate}
\item in networks with diagonalizable adjacency matrices,
\item in networks with adjacency matrices having eigenvalues with high geometric multiplicities.
\end{enumerate}

\section{Conclusions}

This paper was mainly motivated by the work of Stewart \cite{S09} about lattices of synchrony subspaces, and by the work of Kamei \cite{K09} about the relation between synchrony subspaces and classes of eigenvectors of the corresponding adjacency matrix. It was also motivated by our observations, in some worked examples, that on the list of all synchrony subspaces written as direct sums of Jordan subspaces, some Jordan subspaces were essential and some others were not.

We then proved the existence of a class of Jordan subspaces whose elements were sufficient to generate all synchrony subspaces by direct sums and defined the elements in this class as {\em special Jordan subspaces}. To be more precise, we showed that all synchrony subspaces can be obtained through a small set of Jordan subspaces, by direct sums. We also emphasize the close relationship between the special Jordan subspaces of a regular network and the corresponding lattice structure of synchrony subspaces.

\appendix

\section*{Appendix}

We consider four examples of regular networks and, for each case, we list all synchrony subspaces and present the corresponding lattice.

\renewcommand{\thesection}{A}

\begin{ex} \normalfont Consider the 3-cell regular network in Figure \ref{FigValMult} and the eigenspaces of the corresponding adjacency matrix:
$$G_2=\{x_1-2x_2+x_3=0\},  \;G_0=\{x_1=x_3=0\}.$$
Due to Proposition \ref{valency}, to calculate all synchrony subspaces we can assume that the list of all special Jordan subspaces in $G_2$ is finite and without loss of generality we consider that the fully synchrony subspace $F$ and $\mbox{span}\{(2,1,0)\}$ are the unique special Jordan subspaces in this eigenspace. The eigenspace $G_0$ is 1-dimensional and thus it is a special Jordan subspace. Therefore, there is only one nontrivial synchrony subspace, namely, $\{x_1=x_3\}=F\oplus G_0$. The corresponding lattice is presented in Figure \ref{FigLatticesValency} - left.
\end{ex}

\begin{ex} \normalfont Consider the 4-cell regular network in Figure \ref{FigValMult} and the eigenspaces of the corresponding adjacency matrix:
$$G_3=\{x_1-3x_2+x_3+x_4=0\} ,\; G_0=\{x_1=x_3=x_4=0\}.$$
Due to Proposition \ref{valency}, we can assume that the list of all special Jordan subspaces in $G_3$ is finite and, using the calculations of Example \ref{ExValMult}, we can assume that the special Jordan subspaces in $G_3$ are the following:
\begin{enumerate}
\item $W_{3,1}=\{x_1=x_2=x_3=x_4\}$,
\item $W_{3,2}=\{x_1=x_2=0,\, x_3+x_4=0\}$,
\item $W_{3,3}=\{x_1=x_3,\, x_2=0, \, 2x_1+x_4=0\}$,
\item $W_{3,4}=\{x_1=x_4,\, x_2=0, \, 2x_1+x_3=0\}$,
\item $W_{3,5}=\{x_2=x_3=0,\, x_1+x_4=0\}$,
\item $W_{3,6}=\{x_2=x_4=0,\, x_1+x_3=0\}$,
\item $W_{3,7}=\{x_3=x_4, \, x_2=0,\, x_1+2x_3=0\}$.
\end{enumerate} The eigenspace $G_0$ is 1-dimensional and thus it is a special Jordan subspace. Therefore, there are eight special Jordan subspaces to the network and we obtain exactly four nontrivial synchrony subspaces:
$$\begin{array}{ll}
 (1)\; W_{3,1}\oplus G_0= \{x_1=x_3=x_4\}, & (3)\; W_{3,1}\oplus W_{3,3}\oplus G_0 =\{x_1=x_3\},\\
(2) \; W_{3,1}\oplus W_{3,4}\oplus G_0 = \{x_1=x_4\}, & (4)\; W_{3,1}\oplus W_{3,7}\oplus G_0=\{x_3=x_4\}.\end{array}$$
The corresponding lattice is presented in Figure \ref{FigLatticesValency} - right.
\end{ex}

\begin{figure}[!ht]\centering \includegraphics[scale=0.35]{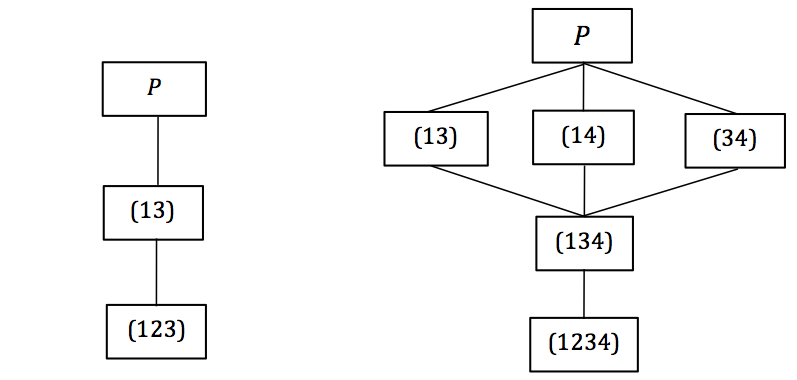} \caption{The lattices of all synchrony subspaces of the networks in Figure \ref{FigValMult}. Each cycle $(i_1,\cdots, i_s)$ denotes the equality of the corresponding cell coordinates and $P$ denotes the total phase space.}\label{FigLatticesValency}\end{figure}

\begin{ex}\normalfont \label{ExNet6} Consider the regular network in Figure \ref{FigNet6} and the eigenspaces of the corresponding adjacency matrix:

\begin{enumerate}\item $G_2= \{x_1=x_2=x_3=x_4=x_5\}$, \item $G_1= \{x_1=x_2=x_4=0, \; x_3=x_5\}$,   \item $G_{-1}= \{x_1+x_3+x_5=0, \; x_1+x_2+x_4=0\}. $\end{enumerate}

\begin{figure}[!ht]
\begin{minipage}[b]{0.37\linewidth}\centering
        \includegraphics[scale=0.4]{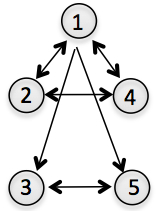} \vspace{-1.3cm}
 \end{minipage}  \begin{minipage}[b]{0.4\linewidth} \centering
$$\left(\begin{array}{ccccc}
0 & 1 & 0 & 1 & 0  \\
1 & 0 & 0 & 1 & 0  \\
1 & 0 & 0 & 0 & 1  \\
1 & 1 & 0 & 0 & 0  \\
1 & 0 & 1 & 0 & 0    \end{array}\right)$$
 \end{minipage}
 \caption{Regular network and the corresponding adjacency matrix of Example \ref{ExNet6}.} \label{FigNet6}
 \end{figure}

There are exactly thirteen special Jordan subspaces to this network:
\begin{enumerate}
\item $G_2 \, =\,\{x_1=x_2=x_3=x_4=x_5 \}$,
\item $G_1 \, =\,\{x_1=x_2=x_4=0, \;x_3=x_5 \}$,
\item $W_{-1,1} \, =\,\{x_1=x_2=x_3, \; x_4=x_5, \; 2x_1+x_4=0\}$,
\item $W_{-1,2} \, = \,\{x_1=x_2=x_4=0, \; x_3+x_5=0\}$,
\item $W_{-1,3} \, =\, \{ x_1=x_2=x_5, \; x_3=x_4, \; 2x_1+x_3=0\}$,
\item $W_{-1,4} \, = \, \{ x_1=x_3=x_4,\; x_2=x_5, \; 2x_1+x_2=0\}$,
\item $W_{-1,5} \, = \, \{ x_1=x_3=x_5=0,\; x_2+x_4=0\}$,
\item $W_{-1,6} \, = \, \{ x_1=x_4= x_5, \; x_2=x_3, \; 2x_1+x_2=0\}$,
\item $W_{-1,7} \, = \,\{x_2=x_3=x_4=x_5, \; x_1+2x_2=0\}$,
\item $W_{-1,8} \, = \, \{ x_1=x_2, \; x_3=x_5, \; x_1+2x_3=0, \; 2x_1+x_4=0\}$,
\item $W_{-1,9} \, = \, \{ x_1=x_3, \; x_2=x_4,\; x_1+2x_2=0, \; 2x_1+x_5=0\}$,
\item $W_{-1,10} \, = \, \{ x_1=x_4, \; x_3=x_5,\; x_1+2x_3=0, \; 2x_1+x_2=0\}$,
\item $W_{-1,11} \, = \, \{ x_1=x_5, \; x_2=x_4,\; x_1+2x_2=0, \; 2x_1+x_3=0\}$.
\end{enumerate}

For each $1\le l\le 4$, we analyze if there are $l$ special Jordan subspaces with $5-l$ common equalities of coordinates and whose sum is $l$-dimensional, obtaining exactly sixteen nontrivial synchrony subspaces:

\begin{enumerate}
\item $W_{2,1} \oplus W_{1,1}= \{x_1=x_2=x_4, \, x_3=x_5\}$,
\item $W_{2,1} \oplus W_{-1,1}= \{x_1=x_2=x_3, \, x_4=x_5\}$,
\item $W_{2,1} \oplus W_{-1,3}= \{x_1=x_2=x_5, \, x_3=x_4\}$,
\item $W_{2,1} \oplus W_{-1,4}= \{x_1=x_3=x_4, \, x_2=x_5\}$,
\item $W_{2,1} \oplus W_{-1,6}= \{x_1=x_4=x_5, \, x_2=x_3\}$,
\item $W_{2,1} \oplus W_{-1,7}= \{x_2=x_3=x_4=x_5\}$,
\item $W_{2,1} \oplus W_{1,1} \oplus W_{-1,2}= \{x_1=x_2=x_4\}$,
\item $W_{2,1} \oplus W_{1,1} \oplus W_{-1,7}= \{x_2=x_4, \, x_3=x_5\}$,
\item $W_{2,1} \oplus W_{1,1} \oplus W_{-1,8}= \{x_1=x_2, \, x_3=x_5\}$,
\item $W_{2,1} \oplus W_{1,1} \oplus W_{-1,10}= \{x_1=x_4, \, x_3=x_5\}$,
\item $W_{2,1} \oplus W_{-1,1} \oplus W_{-1,6}= \{x_2=x_3, \, x_4=x_5\}$,                       
\item $W_{2,1} \oplus W_{-1,3} \oplus W_{-1,4}= \{x_2=x_5, \, x_3=x_4\} $,
\item $W_{2,1} \oplus W_{1,1} \oplus W_{-1,1} \oplus W_{-1,2} = \{x_1=x_2\}$,
\item $W_{2,1} \oplus W_{1,1} \oplus W_{-1,2} \oplus W_{-1,4} = \{x_1=x_4\}$,
\item $W_{2,1} \oplus W_{1,1} \oplus W_{-1,2} \oplus W_{-1,7} = \{x_2=x_4\}$,
\item $W_{2,1} \oplus W_{1,1} \oplus W_{-1,5} \oplus W_{-1,7} = \{x_3=x_5\}$.
\end{enumerate}
The corresponding lattice is presented in Figure \ref{LatticeNet6}. Notice that there are thirteen special Jordan subspaces and there are ten join-irreducible elements in this lattice, namely, all elements of the first and the second level, together with $\{x_1=x_2=x_4\}$, $\{x_1=x_2, x_3=x_5\}$ and $\{x_1=x_4, x_3=x_5\}$.

\begin{figure}[!ht]\centering \includegraphics[scale=0.6]{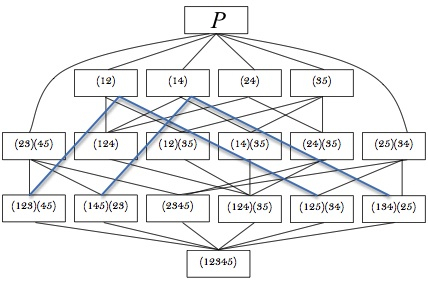} \caption{The lattice of all synchrony subspaces of the network in Figure \ref{FigNet6}. Each cycle $(i_1,\cdots, i_s)$ denotes the equality of the corresponding cell coordinates and $P$ denotes the total phase space.}\label{LatticeNet6}\end{figure}
\end{ex}

\begin{ex}  \normalfont \label{ExMeuEx}Consider the regular network in Figure \ref{FigMeuEx}. The corresponding adjacency matrix has two different eigenvalues, namely, 1 and 0, and the corresponding generalized eigenspaces are:
$$G_1= \{x_1=x_2=x_3=x_4=x_5=x_6\}, \; G_0= \{x_1=0\}.$$

\begin{figure}[!ht]

\begin{minipage}[b]{0.37\linewidth}\centering
        \includegraphics[scale=0.25]{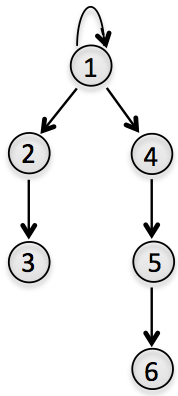} \vspace{-1.3cm}
 \end{minipage}  \begin{minipage}[b]{0.4\linewidth} \centering
$$A= \left(\begin{array}{cccccc}
1 & 0 & 0 & 0 & 0 & 0  \\
1 & 0 & 0 & 0 & 0 & 0  \\
0 & 1 & 0 & 0 & 0 & 0  \\
1 & 0 & 0 & 0 & 0 & 0  \\
0 & 0 & 0 & 1 & 0 & 0  \\
0 & 0 & 0 & 0 & 1 & 0  \\  \end{array}\right)$$
 \end{minipage}
 \caption{Regular network and the corresponding adjacency matrix of Example \ref{ExMeuEx}.} \label{FigMeuEx}
 \end{figure}

$G_1$ is a 1-dimensional eigenspace and thus it is a special Jordan subspace. $G_{0}$ is a generalized eigenspace of order 3 and
$$\begin{array}{l}K^1=\Ker (A)=\{x_1=x_2=x_4=x_5=0\}, \\ K^2=\Ker (A^2)=\{x_1=x_4=0\}.\end{array}$$ 

All 1-dimensional special Jordan subspaces to the network in $G_1$ are precisely all 1-dimensional special subspaces in $K^1$, namely:
\begin{enumerate}
\item $  \, W_{1}= \{x_1=x_2=x_3=x_4=x_5=0\}$,
\item $ \, W_{2}=\{x_1=x_2=x_4=x_5=x_6=0\}$,
\item $ \, W_{3}=\{x_1=x_2= x_4=x_5, x_3=x_6=0\}$.\end{enumerate}

To obtain all 2-dimensional special Jordan subspaces, notice that:
$$K^1\cap \Im (A)  =K^1 \cap \{x_1=x_2=x_4\}=K^1.$$
We calculate all 2-dimensional special Jordan subspaces using Theorem \ref{quebracabecas}. This theorem guarantees the existence of 2-dimensional special Jordan subspaces containing each one of the 1-dimensional special subspaces in $N^2=K^1$. The possible special Jordan subspaces containing $W_1$ are obtained with the calculation of all 2-dimensional special subspaces $J_2$ in the pre-image
$$A^{-1}(W_1)=\{x_1=x_2=x_4=0\},$$ satisfying $P(W_1)\subset P(J_2)$ and $J_2\cap (K^2\backslash K^1)\ne \emptyset$, leading to:
$$\begin{array}{l}
U_1= \{x_1=x_2=x_3=x_4=0\} \; \supset W_1,\\
U_2=\{x_1=x_2=x_4=0, x_3=x_5\} \; \supset W_1.\end{array}$$ Applying the same procedure to the other two pre-images,
$$\begin{array}{l}A^{-1}(W_2)=\{x_1=x_4=x_5=0\}, \\ A^{-1}(W_3)=\{x_1=x_4=0, x_2=x_5\},\end{array}$$
we obtain the following three additional Jordan subspaces:
$$\begin{array}{l}
U_3= \{x_1=x_4=x_5=x_6=0\}\; \supset W_2,\\
U_{4}=\{x_1=x_4=x_5=0, x_2=x_6\}\; \supset W_2,\\
U_{5}=\{x_1=x_4=0,x_2=x_5,x_3=x_6\}\; \supset W_3.\end{array}$$
These subspaces are 2-dimensional special subspaces in $K^2$ and thus, they are special Jordan subspaces to the network.

It is easy to understand that there are no more 2-dimensional special Jordan subspaces. Indeed, if $J$ is a 2-dimensional Jordan subspace, then $P(J)$ contains $P(W_i)$, for some $1\le i\le 3$. If $P(W_1)=\{x_1=x_2=x_3=x_4=x_5\}\subset P(J)$, we must take into account that $K^2=\{x_1=x_4=0\}$ and that $$U_1= \{x_1=x_2=x_3=x_4=0\} , \; U_2=\{x_1=x_2=x_4=0, x_3=x_5\}$$ are special Jordan subspaces satisfying $P(W_1)\subset P(U_j)$, with $1\le j\le 2$. Thus, we just have to prove that if $$P(J)=\{x_1=x_4=x_5\}\, \mbox{ or }\, P(J)=\{x_1=x_4, x_2=x_5\}$$ then $J$ is not a special Jordan subspace. But that trivially follows from the existence of $U_3$ and $U_5$, respectively. A similar situation occurs when $P(W_2)\subset P(J)$ and $P(W_3)\subset P(J)$.

To obtain all 3-dimensional special Jordan subspaces, notice that:
$$ K^2\cap \Im A= \{x_1=x_4=0\}\cap \{x_1=x_2=x_4\}=\{x_1=x_2=x_4=0\}.$$
Therefore, calculating all 3-dimensional special subspaces $J_3$ in the pre-images:
$$A^{-1}(U_1)=\{x_1=x_2=0\}\; \mbox{ and } \; A^{-1}(U_2)=\{x_1=0, x_2=x_4\},$$ satisfying $P(U_1)\subset P(J_3)$ and $P(U_2)\subset P(J_3)$, respectively, and $J_3\cap (G_0\backslash K^2)\ne \emptyset$, we obtain the following three 3-dimensional Jordan subspaces:
$$\begin{array}{l} V_{1}=\{x_1=x_2=x_3=0\},\\
V_2=\{x_1=x_2=0,x_3=x_4\},\\
V_3=\{x_2=x_4,x_3=x_5, x_1=0\}.\end{array}$$

Analogously, it is proved that these are 3-dimensional special Jordan subspaces and that they are the unique subspaces in this condition.

Hence, we obtain the following list of special Jordan subspaces:
\begin{enumerate}
\item $G_1= \{x_1=x_2=x_3=x_4=x_5=x_6\}$,
\item $W_{1}= \{x_1=x_2=x_3=x_4=x_5=0\}$,
\item $W_{2}=\{x_1=x_2=x_4=x_5=x_6=0\}$,
\item $W_{3}=\{x_1=x_2= x_4=x_5, x_3=x_6=0\}$,
\item $U_1= \{x_1=x_2=x_3=x_4=0\}$,
\item $U_4=\{x_1=x_2=x_4=0, x_3=x_5\}$,
\item $U_8= \{x_1=x_4=x_5=x_6=0\}$,
\item $U_{10}=\{x_1=x_4=x_5=0, x_2=x_6\}$,
\item $U_{15}=\{x_1=x_4=0,x_2=x_5,x_3=x_6\}$,
\item $V_{1}=\{x_1=x_2=x_3=0\}$,
\item $V_2=\{x_1=x_2=0,x_3=x_4\}$,
\item $V_3=\{x_2=x_4,x_3=x_5, x_1=0\}$.\end{enumerate}

For each $1\le l\le 5$, we analyze if there are special Jordan subspaces with $6-l$ common equalities of coordinates and whose sum is $l$-dimensional, obtaining exactly 18 nontrivial synchrony subspaces:
\begin{enumerate}
\item $G_1 \oplus W_1= \{x_1=x_2=x_3=x_4=x_5\}$,
\item $G_1 \oplus W_2= \{x_1=x_2=x_4=x_5=x_6\}$,
\item $G_1  \oplus W_3= \{x_1=x_2=x_4=x_5, x_3=x_6\}$,
\item $G_1 \oplus W_1 \oplus W_2= \{x_1=x_2=x_4=x_5\}$,
\item $G_1 \oplus U_1= \{x_1=x_2=x_3=x_4\}$,
\item $G_1 \oplus U_2 = \{x_1=x_2=x_4, x_3=x_5\}$,
\item $G_1\oplus  U_8= \{x_1=x_4=x_5=x_6\}$,
\item $G_1 \oplus U_{10}= \{x_1=x_4=x_5, x_2=x_6\}$,
\item $G_1 \oplus U_{15}= \{x_1=x_4, x_2=x_5,  x_3=x_6\}$,
\item $G_1  \oplus W_{2}\oplus U_1= \{x_1=x_2=x_4\}$,
\item $G_1 \oplus W_2 \oplus U_8= \{x_1=x_4=x_5\}$,
\item $G_1 \oplus W_1\oplus U_{15}= \{x_1=x_4, x_2=x_5\}$,
\item $G_1 \oplus V_1=\{x_1=x_2=x_3\}$,
\item $G_1\oplus V_2= \{x_1=x_2, x_3=x_4\}$,
\item $G_1 \oplus V_3= \{x_2=x_4, \, x_3=x_5\}$,
\item $G_1 \oplus W_{2}\oplus V_1= \{x_1=x_2\}$,
\item $G_1  \oplus U_1\oplus U_8= \{x_1=x_4\}$,
\item $G_1 \oplus W_{2} \oplus V_{3}= \{x_2=x_4\}$.
\end{enumerate}

The lattice of all synchrony subspaces is presented in Figure \ref{FigLatticeMeuEx}. Notice that in this case, the number of special Jordan subspaces and of join-irreducible elements of the lattice coincides.

\end{ex}
\begin{figure}[!ht]\centering \includegraphics[scale=0.3]{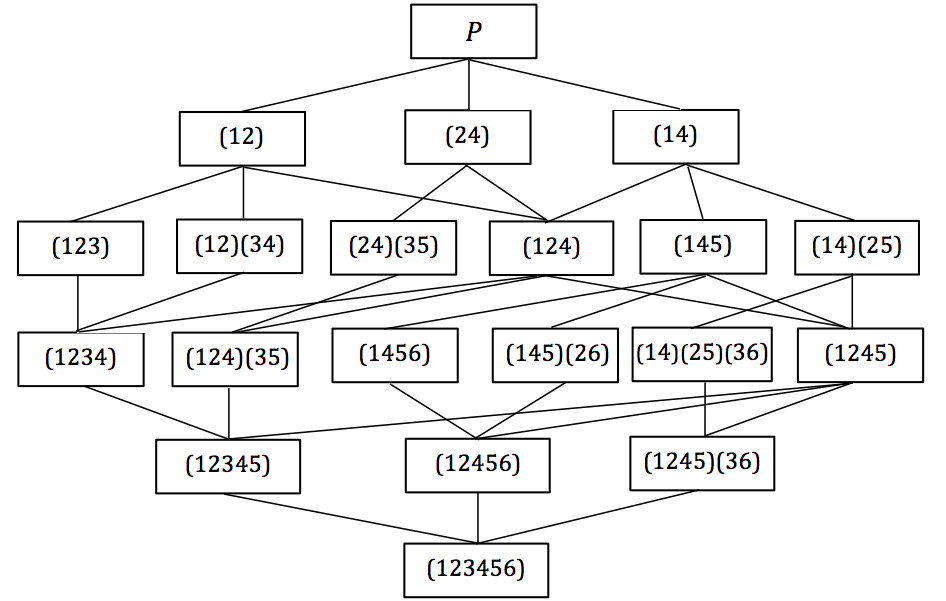} \caption{The lattice of all synchrony subspaces of the network in Figure \ref{FigMeuEx}. Each cycle $(i_1,\cdots, i_s)$ denotes the equality of the corresponding cell coordinates and $P$ denotes the total phase space.}\label{FigLatticeMeuEx}\end{figure}

\end{document}